\documentclass[12pt]{amsart}
\usepackage{}
\usepackage{tikz}
  \usetikzlibrary{matrix,arrows,positioning,shapes.geometric}
\usepackage{amsmath}
\usepackage{amsfonts}
\usepackage{amssymb}
\usepackage[all]{xy}           

\usepackage{bbding}
\usepackage{txfonts}
\usepackage{amscd}

\usepackage[shortlabels]{enumitem}
\usepackage{ifpdf}
\ifpdf
  \usepackage[colorlinks,final,backref=page,hyperindex]{hyperref}
\else
  \usepackage[colorlinks,final,backref=page,hyperindex,hypertex]{hyperref}
\fi
\usepackage{tikz}
\usepackage[active]{srcltx}

\makeatletter

\newtheorem{thm}{Theorem}[section]
\newtheorem{lem}[thm]{Lemma}

\newtheorem{pro}[thm]{Proposition}

\newtheorem{rmk}[thm]{Remark}
\newtheorem{defi}[thm]{Definition}

\newcommand {\emptycomment}[1]{}

\newcommand{\lon }{\,\rightarrow\,}
\newcommand{\be }{\begin{equation}}
\newcommand{\ee }{\end{equation}}

\newcommand{\g}{\mathfrak g}
\newcommand{\h}{\mathfrak h}

\newcommand{\huaG}{\mathcal{G}}

\newcommand{\huaC}{{\mathcal{C}}}
\newcommand{\huaD}{\mathcal{D}}

\newcommand{\Id}{\rm{Id}}

\newcommand{\br}[1]{   [ \cdot,    \cdot  ]   }

\newcommand{\dM}{\mathrm{d}}

\newcommand{\Hom}{\mathrm{Hom}}

\newcommand{\gl}{\mathfrak {gl}}

\newcommand{\Ker}{\mathrm{Ker}}

\newcommand{\Img}{\mathrm{Im}}

\newcommand{\sgn}{\mathrm{sgn}}

\newcommand{\reg}{\mathrm{reg}}

\begin{document}

\title[]{Cohomologies of pre-LieDer pairs and applications}

\author{Shanshan Liu}
\address{School of Mathematics and Statistics, Northeast Normal University, Changchun 130024, Jilin, China}
\email{shanshanmath@163.com}

\author{Liangyun Chen*}
\address{School of Mathematics and Statistics, Northeast Normal University, Changchun 130024, China}
\email{chenly640@nenu.edu.cn}


\begin{abstract}
In this paper,  we use the higher derived bracket to give the controlling algebra of pre-LieDer pairs. We give the cohomology of pre-LieDer pairs by using the twist $L_\infty$-algebra of this controlling algebra. In particular, we define the cohomology of regular pre-LieDer pairs. We study infinitesimal deformations of pre-LieDer pairs, which are characterized by the second cohomology group of pre-LieDer pairs. We also define the cohomology of regular pre-LieDer pairs with coefficients in arbitrary representation and using the second cohomology group to classify abelian extensions of regular pre-LieDer pairs.
\end{abstract}

\renewcommand{\thefootnote}{\fnsymbol{footnote}}
\footnote[0]{* Corresponding author}
\keywords{pre-LieDer pair, $L_{\infty}$-algebra, $V$-data, cohomology, deformation, extension }
\footnote[0]{{\it{MSC 2020}}: 17A36, 17A40, 17B10, 17B40, 17B60, 17B63, 17D25}

\maketitle
\vspace{-5mm}
\tableofcontents

\allowdisplaybreaks


\section{Introduction}
Derivations play an important role in the study of various algebraic structures. Derivations can be used to construct higher derived bracket and homotopy Lie algebras \cite{TTV}, deformation formulas \cite{VEC} and difffferential Galois theory \cite{ARM}. They are also important in control theory \cite{VA1,VA2}, functional analysis \cite{AFB1,AFB2}. Moreover, they also play a fundamental role in the study of  gauge theories in quantum field theory via the BV-formalism \cite{IAB}. In \cite{MD,LJ}, the authors studied the operad of associative algebras with derivation. Recently, in \cite{TR} the authors study the cohomology, extensions and deformations of Lie algebras with derivations. Using similar ideas, the author study associative algebras with derivations \cite{AD1} and Leibniz algebras with derivations \cite{AD2}.

The notion  of a pre-Lie algebra (also called  left-symmetric algebras, quasi-associative algebras, Vinberg algebras and so on) has been introduced independently  by M. Gerstenhaber in  deformation theory  of rings and algebras \cite{Gerstenhaber63}. Pre-Lie algebra arose from the study of affine manifolds and affine structures on Lie group \cite{JLK}, homogeneous convex  cones \cite{Vinberg}. Its defining identity is weaker than associativity. This algebraic structure describes some properties of cochains space in Hochschild cohomology of an associative algebra, rooted trees and vector fields on affine spaces.
 Moreover, it is playing an increasing role in algebra, geometry and physics due to their applications in nonassociative algebras, combinatorics,  numerical Analysis and quantum field theory, see also in \cite{BD,Bai,Bai2,CK}. There is a close relationship between pre-Lie algebras and Lie algebras: a pre-Lie algebra $(\g,\cdot)$ gives rise to a Lie algebra $(\g,[\cdot,\cdot]_C)$ via the commutator bracket, which is called the subadjacent Lie algebra and denoted by ${\g}^C$. Furthermore, the map $L:\g\longrightarrow\gl(\g)$, defined by $L_xy=x\cdot y$ for all $x,y\in \g$, gives rise to a representation of the subadjacent Lie algebra ${\g}^C$ on $\g$.

The purpose of the paper is to study the cohomology of a pre-LieDer pair and its applications. The motivation for such a study due to the study of relative difference Lie algebras and relative Rota-Baxter Lie algebras \cite{Jiang1,Jiang2,LA}. We use the higher derived bracket to give the controlling algebra of pre-LieDer pairs. Using the Maurer-Cartan element of this controlling algebra, we construct a twist $L_\infty$-algebra. We give the cohomology of pre-LieDer pairs by using this twist $L_\infty$-algebra and study infinitesimal deformations of pre-LieDer pairs. We also define the cohomology of regular pre-LieDer pairs with coefficients in arbitrary representation and study extensions of regular pre-LieDer pairs.

The paper is organized as follows. In Section \ref{sec:controlling-derivation}, first, we give the notion of pre-LieDer pairs. Then, by a given representation $(V;\rho,\mu)$ of a pre-Lie algebra $(\g,\pi)$, we recall the graded Lie algebra whose Maurer-Cartan elements are $\pi+\rho+\mu$. Finally, we use this Maurer-Cartan element to give the cohomology of pre-Lie algebras with representations. In Section \ref{sec:deformation}, first, we recall the notion of $L_{\infty}$-algebras and higher derived brackets. Then, we use higher derived brackets to construct an $L_{\infty}$-algebra, whose  Maurer-Cartan elements are pre-LieDer pairs. Finally, using this Maurer-Cartan element, we construct a twist $L_\infty$-algebra, which controls deformations of pre-LieDer pairs.  In Section \ref{sec:infinitesimal-deformation}, first, using the twist $L_{\infty}$-algebra constructed from last section, we give the cohomology of pre-LieDer pairs. Then, we study infinitesimal deformations of pre-LieDer pairs by using this cohomology. We show that equivalent infinitesimal deformations are in the same second cohomology group. Finally, we use this cohomology to give the  cohomology of regular pre-LieDer pairs. In Section \ref{sec:abelian-extension}, first, we introduce the notion of representations of regular pre-LieDer pairs. Then, we give cohomologies of regular pre-LieDer pairs with coefficients in an arbitrary representation. Finally, we deal with abelian extensions of regular pre-LieDer pairs. We show that the second cohomology group classifies abelian extensions of  regular pre-LieDer pairs.

In this paper, we work over an algebraically closed field $\mathbb{K}$ of characteristic $0$ and all the vector spaces are over $\mathbb{K}$.
\vspace{2mm}
\noindent
\section{Cohomologies of pre-Lie algebras with representations}\label{sec:controlling-derivation}
In this section, we define a pre-LieDer pair, which consists of a pre-Lie algebra $\g$ and a derivation $D$. We give a representation $(V;\rho,\mu)$ of a pre-Lie algebra $(\g,\pi)$ and recall the graded Lie algebra whose Maurer-Cartan elements are $\pi+\rho+\mu$. Then we use this Maurer-Cartan element to give the cohomology of pre-Lie algebras with representations.
\begin{defi}{\rm(\cite{BD})}
A {\bf pre-Lie algebra} $(\g,\cdot)$ is a vector space $\g$ equipped with a bilinear product $\cdot:\g\otimes \g\longrightarrow \g$, such that for all $x,y,z\in \g$, the following equality is satisfied:
\begin{equation*}
(x\cdot y)\cdot z-x\cdot (y\cdot z)=(y\cdot x)\cdot z-y\cdot (x\cdot z).
\end{equation*}
\end{defi}
Let $(\g,\cdot)$ be a pre-Lie algebra. The commutator $[x,y]_C=x\cdot y-y\cdot x$ gives a Lie algebra $(\g,[\cdot,\cdot]_C)$, which is denoted by ${\g}^C$ and called  the {\bf sub-adjacent Lie algebra} of $(\g,\cdot)$.
\begin{defi}{\rm(\cite{Bai})}
 A {\bf representation} of a pre-Lie algebra $(\g,\cdot)$ on a vector space $V$ consists of a pair $(\rho,\mu)$, where $\rho:\g\longrightarrow \gl(V)$ is a representation of the sub-adjacent Lie algebra ${\g}^C$ on $V$, and $\mu:\g\longrightarrow \gl(V)$ is a linear map, such that for all $x,y\in \g$:
\begin{equation*}
\mu(y)\circ\mu(x)-\mu(x\cdot y)=\mu(y)\circ\rho(x)-\rho(x)\circ\mu(y).
\end{equation*}
\end{defi}
We denote a representation of a pre-Lie algebra $(\g,\cdot)$ by $(V;\rho,\mu)$. Furthermore, let $L,R:\g\longrightarrow \gl(\g)$ be linear maps, where $L_xy=x\cdot y, R_xy=y\cdot x$. Then $(\g;L,R)$ is also a representation, which is called the regular representation.
\begin{defi}
Let $(V;\rho,\mu)$ be a representation of a pre-Lie algebra $(\g,\cdot)$ and $D:\g\longrightarrow V$ a linear map such that for all $x,y\in \g$
\begin{equation*}
D(x\cdot y)=\rho(x)D(y)+\mu(y)D(x).
\end{equation*}
Then $D$ is a {\bf derivation} of the pre-Lie algebra $(\g,\cdot)$ with respect to the representation $(V;\rho,\mu)$. A pre-Lie algebra $(\g,\cdot)$ with $D$ is called a {\bf pre-LieDer pair}, which is denoted by $(\g,D,\rho,\mu)$. In particular, let $(\g;L,R)$ be a regular representation, then $(\g,D,L,R)$ is called a regular pre-LieDer pair and simply denoted by $(\g,D)$.
\end{defi}

\begin{defi}
Let $(\g,D,\rho,\mu)$ and $(\g',D',\rho',\mu')$ be two pre-LieDer pairs. A {\bf morphiam} $(f_{\g},f_{V})$ from $(\g,D,\rho,\mu)$ to $(\g',D',\rho',\mu')$ consists of a pre-Lie algebra morphism $f_{\g}:\g\longrightarrow \g'$ and a linear map $f_{V}:V\longrightarrow V'$ such that for all $x\in \g$
\begin{eqnarray}
\label{mor-1}f_{V}\circ \rho(x)&=&\rho'(f_{\g}(x))\circ f_{V},\\
\label{mor-2}f_{V}\circ \mu(x)&=&\mu'(f_{\g}(x))\circ f_{V},\\
\label{mor-3}f_{V}\circ D&=&D'\circ f_{\g}.
\end{eqnarray}
\end{defi}

Let $(V;\rho,\mu)$ be a representation of a pre-Lie algebra $(\g,\cdot)$. The set of $n$-cochains is given by
\begin{equation*}
 C^n(\g;V)=\Hom(\wedge^{n-1} \g\otimes \g,V),\quad \forall n\geq 1.
\end{equation*}
For all $f \in C^n(\g;V),~ x_1,\dots,x_{n+1} \in \g$, define the coboundary operator
$\dM:C^n(\g;V)\longrightarrow C^{n+1}(\g;V)$ by
\begin{eqnarray*}
(\dM f)(x_1,\dots,x_{n+1})&=&\sum_{i=1}^n(-1)^{i+1}\rho(x_i)f(x_1,\dots,\hat{x_i},\dots,x_{n+1})\\
&&+\sum_{i=1}^n(-1)^{i+1}\mu(x_{n+1})f(x_1,\dots,\hat{x_i},\dots,x_n,x_i)\\
&&-\sum_{i=1}^n(-1)^{i+1}f(x_1,\dots,\hat{x_i}\dots,x_n,x_i\cdot x_{n+1})\\
&&+\sum_{1\leq i<j\leq n}(-1)^{i+j} f([x_i,x_j]_C,x_1,\dots,\hat{x_i},\dots,\hat{x_j},\dots,x_{n+1}).
\end{eqnarray*}

\begin{thm}{\rm(\cite{DA})}
With the above notation, $(\oplus_{n=1}^{+\infty}C^n(\g;V),\dM)$ is a chain complex.
\end{thm}

\begin{defi}{\rm(\cite{DA})}
The cohomology of the cochain complex $(\oplus_{n=1}^{+\infty}C^n(\g;V),\dM)$ is called the cohomology of the pre-Lie algebra $(\g,\cdot_{\g})$ with coefficients in the representation $(V;\rho,\mu)$. The corresponding cohomology group is denoted by $H^n(\g;V)$.
\end{defi}
\begin{rmk}
We use $\dM_{\reg}$ to denote the coboundary operator of $(\g,\cdot)$ with the coefficients in the regular representation.
\end{rmk}

A permutation $\sigma\in \mathbb{S}_n$ is called an $(i,n-i)$-unshuffle if $\sigma(1)<\dots<\sigma(i)$ and $\sigma(i+1)<\dots<\sigma(n)$. If $i=0$ and $i=n$, we assume $\sigma=\Id$. The set of all $(i,n-i)$-unshuffles will be denoted by $\mathbb{S}_{(i,n-i)}$. The notion of an $(i_1,\dots,i_k)$-unshuffle and the set $\mathbb{S}_{(i_1,\dots,i_k)}$ are defined similarly.

Let $\g$ be a vector space. We consider the graded vector space $C^*(\g;\g)=\oplus_{n=1}^{+\infty}C^n(\g;\g)=\oplus_{n=1}^{+\infty}\Hom(\wedge^{n-1}\g\otimes \g,\g)$. It was shown in \cite{FC,AN,QW} that $C^*(\g;\g)$ equipped with the Matsushima-Nijenhuis bracket
\begin{equation*}
[P,Q]^{MN}=P\circ Q-(-1)^{pq}Q\circ P,\quad \forall P\in C^{p+1}(\g;\g), Q\in C^{q+1}(\g;\g)
\end{equation*}
ia a graded Lie algebra, where $P\circ Q\in C^{p+q+1}(\g;\g)$ is defined by
\begin{eqnarray*}
 &&P\circ Q(x_1,\dots,x_{p+q+1})\\
&=&\sum_{\sigma\in \mathbb{S}(q,1,p-1)}\sgn(\sigma)P(Q(x_{\sigma(1)},\dots,x_{\sigma(q)},x_{\sigma(q+1)}),x_{\sigma(q+2)},\dots,x_{\sigma(p+q)},x_{p+q+1})\\
&&+(-1)^{pq}\sum_{\sigma\in \mathbb{S}(p,q)}\sgn(\sigma)P(x_{\sigma(1)},\dots,x_{\sigma(p)},Q(x_{\sigma(p+1)},\dots,x_{\sigma(p+q)},x_{p+q+1})).
\end{eqnarray*}
In particular, $\pi\in \Hom(\otimes^2\g,\g)$ defines a pre-Lie algebra if and only if $[\pi,\pi]^{MN}=0$. If $\pi$ is a pre-Lie algebra structure, then $d_{\pi}:=[\pi,\cdot]^{MN}$ is a graded derivation of the graded Lie algebra $(C^*(\g;\g),[\cdot,\cdot]^{MN})$ satisfying $d_{\pi}\circ d_{\pi}=0$, so that $(C^*(\g;\g),[\cdot,\cdot]^{MN},d_{\pi})$ becomes a differential graded Lie algebra.

Let $\g_1$ and $\g_2$ be vector spaces and elements in $\g_1$ will be denoted by $x,y,x_i$ and elements in $\g_2$ will be denoted by $u,v,v_i$. Let $c:\wedge^{n-1}\g_1\otimes \g_1\longrightarrow \g_2$ be a linear map. We can construct a linear map $\hat{c}\in C^n(\g_1\oplus \g_2,\g_1\oplus \g_2)$ by
\begin{equation*}
\hat{c}(x_1+v_1,\dots,x_n+v_n):=c(x_1,\dots,x_n).
\end{equation*}
In general, for a given linear map $f:\wedge^{k-1}\g_1\otimes\wedge^l\g_2\otimes\g_1\longrightarrow\g_j$ for $j\in \{1,2\}$, we define a linear map $\hat{f}\in C^{k+l}(\g_1\oplus \g_2,\g_1\oplus \g_2)$ by
\begin{equation*}
\hat{f}(x_1+v_1,\dots,x_{k+l}+v_{k+l})=\sum_{\sigma\in \mathbb{S}(k-1,l)}\sgn(\sigma)f(x_{\sigma(1)},\dots,x_{\sigma(k-l)},v_{\sigma(k)},\dots,v_{\sigma(k+l-1)},x_{k+l}).
\end{equation*}
Similarly, for $f:\wedge^k\g_1\otimes\wedge^{l-1}\g_2\otimes\g_2\longrightarrow\g_j$ for $j\in \{1,2\}$, we define a linear map $\hat{f}\in C^{k+l}(\g_1\oplus \g_2,\g_1\oplus \g_2)$ by
\begin{equation*}
\hat{f}(x_1+v_1,\dots,x_{k+l}+v_{k+l})=\sum_{\sigma\in \mathbb{S}(k,l-1)}\sgn(\sigma)f(x_{\sigma(1)},\dots,x_{\sigma(k)},v_{\sigma(k+1)},\dots,v_{\sigma(k+l-1)},v_{k+l}).
\end{equation*}
We call the linear map $\hat{f}$ a {\bf horizontal lift} of $f$, or simply a lift.
We define $\huaG^{k,l}=\wedge^{k-1}\g_1\otimes\wedge^l\g_2\otimes\g_1+\wedge^k\g_1\otimes\wedge^{l-1}\g_2\otimes\g_2$. The vector space $\wedge^{n-1}(\g_1\oplus\g_2)\otimes(\g_1\oplus\g_2)$ is isomorphic to the direct sum of  $\huaG^{k,l}, k+l=n$.
\begin{defi}{\rm(\cite{Liu})}
A linear map $f\in \Hom(\wedge^{n-1}(\g_1\oplus\g_2)\otimes(\g_1\oplus\g_2),\g_1\oplus\g_2)$ has a {\bf bidegree} $k|l$ if the following four conditions hold:
\begin{itemize}
\item [$\rm(i)$]  $k+l+1=n$;
\item[$\rm(ii)$] If X is an element in $\huaG^{k+1,l}$, then $f(X)\in \g_1$;
\item[$\rm(iii)$] If X is an element in $\huaG^{k,l+1}$, then $f(X)\in \g_2$;
\item[$\rm(iv)$] All the other case, $f(X)=0$.
\end{itemize}
We denote a linear map $f$ with bidegree $k|l$ by $||f||=k|l$.
\end{defi}
We call a linear map $f$ {\bf homogeneous} if $f$ has a bidegree. We denote the set of homogeneous linear maps of bidegree $k|l$ by $C^{k|l}(\g_1\oplus \g_2,\g_1\oplus \g_2)$. We have $k+l\geq 0,k,l\geq -1$  because $n\geq 1$ and $k+1,l+1\geq 0$.

In our later study, the subspaces $C^{k|0}(\g_1\oplus \g_2,\g_1\oplus \g_2)$ and $C^{l|-1}(\g_1\oplus \g_2,\g_1\oplus \g_2)$ will be frequently used. By the above lift, we have the following isomorphisms:
\begin{eqnarray*}
C^{k|0}(\g_1\oplus \g_2,\g_1\oplus \g_2)&\cong& \Hom(\wedge^k\g_1\otimes\g_1,\g_1)\oplus\Hom(\wedge^k\g_1\otimes\g_2,\g_2)\oplus\Hom(\wedge^{k-1}\g_1\otimes\g_2\otimes\g_1,\g_2);\\
C^{l|-1}(\g_1\oplus \g_2,\g_1\oplus \g_2)&\cong& \Hom(\wedge^{l-1}\g_1\otimes\g_1,\g_2).
\end{eqnarray*}
\begin{lem}{\rm(\cite{Liu})}\label{bidegree-1}
If $||f||=k_f|l_f$ and $||g||=k_g|l_g$, then $[f,g]^{MN}$ has the bidegree $k_f+k_g|l_f+l_g$.
\end{lem}
\begin{lem}{\rm(\cite{Liu})}\label{bidegree-2}
If $||f||=-1|k$ (resp. $k|-1$) and $||g||=-1|l$ (resp. $l|-1$), then $[f,g]^{MN}=0$.
\end{lem}
By Lemma \ref{bidegree-1} and  Lemma \ref{bidegree-2}, we obtain that $(\oplus_{k=0}^{+\infty}C^{k|0}(\g_1\oplus \g_2,\g_1\oplus \g_2), [\cdot,\cdot]^{MN})$ is a graded Lie subalgebra of the graded Lie algebra $(C^*(\g_1\oplus \g_2,\g_1\oplus \g_2), [\cdot,\cdot]^{MN})$ and $(\oplus_{l=0}^{+\infty}C^{l|-1}(\g_1\oplus \g_2,\g_1\oplus \g_2), [\cdot,\cdot]^{MN})$ is an abelian graded Lie subalgebra of the graded Lie algebra $(C^*(\g_1\oplus \g_2,\g_1\oplus \g_2), [\cdot,\cdot]^{MN})$.
\begin{pro}{\rm(\cite{Liu1})}\label{MC-prelierep}
Let $\g$ and $V$ be two vector spaces. Then $\pi+\rho+\mu \in C^{1|0}(\g\oplus V,\g\oplus V)$ is a Maurer-Cartan element of the graded Lie algebra $(\oplus_{k=0}^{+\infty}C^{k|0}(\g\oplus V,\g\oplus V), [\cdot,\cdot]^{MN})$ if and only if $(V;\rho,\mu)$ is a representation of the pre-Lie algebra $(\g,\pi)$.
\end{pro}
Let $(V;\rho,\mu)$ be a representation of a pre-Lie algebra $(\g,\pi)$. By Proposition \ref{MC-prelierep}, $\pi+\rho+\mu$ is a Maurer-Cartan element of the the graded Lie algebra $(\oplus_{k=0}^{+\infty}C^{k|0}(\g\oplus V,\g\oplus V), [\cdot,\cdot]^{MN})$. It follows from the graded Jacobi identity that $\dM_{\pi+\rho+\mu}:=[\pi+\rho+\mu,\cdot]^{MN}$ is a graded derivation of the graded Lie algebra $(\oplus_{k=0}^{+\infty}C^{k|0}(\g\oplus V,\g\oplus V), [\cdot,\cdot]^{MN})$ satisfying $\dM_{\pi+\rho+\mu}^2=0$. Thus, we have
\begin{thm}
Let $(V;\rho,\mu)$ be a representation of a pre-Lie algebra $(\g,\pi)$. Then $(\oplus_{k=0}^{+\infty}C^{k|0}(\g\oplus V,\g\oplus V), [\cdot,\cdot]^{MN},\dM_{\pi+\rho+\mu})$ is a differential graded Lie algebra.

 Furthermore, $(V;\rho+\rho',\mu+\mu')$ is representation of a pre-Lie algebra $(g,\pi+\pi')$ for $\pi'\in \Hom(\g\otimes\g,\g)$, $\rho'\in \Hom(\g\otimes V, V)$ and $\mu'\in \Hom(V\otimes \g,\g)$ if and only if $\pi'+\rho'+\mu'$ is a Maurer-Cartan element of the differential graded Lie algebra $(\oplus_{k=0}^{+\infty}C^{k|0}(\g\oplus V,\g\oplus V), [\cdot,\cdot]^{MN},\dM_{\pi+\rho+\mu})$.
\end{thm}
Let $(V;\rho,\mu)$ be a representation of a pre-Lie algebra $(\g,\pi)$. Define the set of $0$-cochains $C^0(\g,\pi,\rho,\mu)$ to be $0$. For $n\geq 1$, we define the set of $n$-cochains $C^n(\g,\pi,\rho,\mu)$ by
\begin{equation*}
C^n(\g,\pi,\rho,\mu):=\Hom(\wedge^{n-1}\g\otimes\g,\g)\oplus\Hom(\wedge^{n-1}\g\otimes V,V)\oplus\Hom(\wedge^{n-2}\g\otimes V\otimes\g,V).
\end{equation*}
Define the coboundary operator $\partial:C^n(\g,\pi,\rho,\mu)\longrightarrow C^{n+1}(\g,\pi,\rho,\mu)$ by
\begin{equation*}
\partial f:=(-1)^{n-1}\dM_{\pi+\rho+\mu}f=(-1)^{n-1}[\pi+\rho+\mu,f]^{MN}.
\end{equation*}
Since $\dM_{\pi+\rho+\mu}^2=0$, we obtain that $\partial^2=0$. Thus, $(\oplus_{n=0}^{+\infty}C^n(\g,\pi,\rho,\mu),\partial)$ is a cochain complex.
\begin{defi}
The cohomology of the cochain complex $(\oplus_{n=1}^{+\infty}C^n(\g,\pi,\rho,\mu),\partial)$ is called the cohomology of the pre-Lie algebra $(\g,\pi)$ with the representation $(V;\rho,\mu)$. The corresponding cohomology group is denoted by $H^n(\g,\pi,\rho,\mu)$.
\end{defi}
For all $f=(f_{\g},f_{\rho},f_{\mu})\in  C^n(\g,\pi,\rho,\mu)$, where $f_{\g}\in \Hom(\wedge^{n-1}\g\otimes\g,\g)$, $f_{\rho}\in \Hom(\wedge^{n-1}\g\otimes V,V)$ and $f_{\mu}\in \Hom(\wedge^{n-2}\g\otimes V\otimes\g,V)$. Then have
\begin{equation*}
\partial f=((\partial f)_{\g},(\partial f)_{\rho},(\partial f)_{\mu}).
\end{equation*}
By straightforward computation, we obtain that
\begin{eqnarray}
\label{cohomology-1}(\partial f)_{\g}&=&(-1)^{n-1}[\pi,f_{\g}]^{MN}=\dM_{\reg} f_{\g},\\
(\partial f)_{\rho}&=&(-1)^{n-1}([\rho,f_{\g}]^{MN}+[\pi+\rho,f_{\rho}]^{MN})\\
(\partial f)_{\mu}&=&(-1)^{n-1}([\mu,f_{\g}]^{MN}+[\mu,f_{\rho}]^{MN}+[\pi+\rho+\mu,f_{\mu}]^{MN}).
\end{eqnarray}
More precisely, for all $x_1,\dots,x_n \in \g, u\in V$, we have
\begin{eqnarray*}
(\partial f)_{\rho}(x_1,\dots,x_n,u)&=&\sum_{i=1}^n(-1)^{i+1}\rho(f_{\g}(x_1,\dots,\hat{x_i},\dots,x_n,x_i))u\\
&&+\sum_{i=1}^n(-1)^{i+1}\rho(x_i)f_{\rho}(x_1,\dots,\hat{x_i},\dots,x_n,u)\\
&&-\sum_{i=1}^n(-1)^{i+1}f_{\rho}(x_1,\dots,\hat{x_i}\dots,x_n,\rho(x_i)u)\\
&&+\sum_{1\leq i<j\leq n}(-1)^{i+j} f_{\rho}([x_i,x_j]_C,x_1,\dots,\hat{x_i},\dots,\hat{x_j},\dots,x_n,u).
\end{eqnarray*}
and
\begin{eqnarray*}
&&(\partial f)_{\mu}(x_1,\dots,x_{n-1},u,x_n)\\
&=&(-1)^{n-1}(\mu(f_{\g}(x_1,\dots,x_n))u+\mu(x_n)f_{\rho}(x_1,\dots,x_{n-1},u)-f_{\rho}(x_1,\dots,x_{n-1},\mu(x_n)u))\\
&&-\sum_{i=1}^{n-1}(-1)^{i+1}f_{\mu}(x_1,\dots,\hat{x_i},\dots,x_{n-1},u,x_i\cdot x_n)\\
&&+\sum_{i=1}^{n-1}(-1)^{i+1}\rho(x_i)f_{\mu}(x_1,\dots,\hat{x_i},\dots,x_{n-1},u,x_n)\\
&&-\sum_{i=1}^{n-1}(-1)^{i+1}f_{\mu}(x_1,\dots,\hat{x_i}\dots,x_{n-1},\rho(x_i)u,x_n)\\
&&+\sum_{i=1}^{n-1}(-1)^{i+1}\mu(x_n)f_{\mu}(x_1,\dots,\hat{x_i},\dots,x_{n-1},u,x_i)\\
&&+\sum_{i=1}^{n-1}(-1)^{i+1}f_{\mu}(x_1,\dots,\hat{x_i},\dots,x_{n-1},\mu(x_i)u,x_n)\\
&&+\sum_{1\leq i<j\leq n-1}(-1)^{i+j} f_{\mu}([x_i,x_j]_C,x_1,\dots,\hat{x_i},\dots,\hat{x_j},\dots,x_{n-1},u,x_n).
\end{eqnarray*}

\section{Maurer-Cartan characterizations and deformations of pre-LieDer pairs}\label{sec:deformation}
In this section, first, we recall $L_{\infty}$-algebras and higher derived brackets. Then, we use higher derived brackets to construct an $L_{\infty}$-algebra, whose  Maurer-Cartan elements are pre-LieDer pairs. Finally, we construct a twist $L_\infty$-algebra by this Maurer-Cartan element, which controls deformations of  pre-LieDer pairs.
\subsection{$L_{\infty}$-algebras and higher derived brackets}
\
\newline
\indent
Let $\g=\oplus_{k\in\mathbb Z}\g^k$ be a $\mathbb Z$-graded vector space. The desuspension operator $s^{-1}$ changes the grading of $\g$ according to the rule $(s^{-1}\g)^i:=\g^{i+1}$. The degree $-1$ map $s^{-1}:\g\longrightarrow s^{-1}\g$ is defined by sending $x\in \g$ to its copy $s^{-1}x\in s^{-1}\g$.

The notion of an $L_{\infty}$-algebra was introduced by Stasheff in \cite{JS}. See \cite{LT,LT1} for more details.
\begin{defi}
An {\bf  $L_\infty$-algebra} is a $\mathbb Z$-graded vector space $\g=\oplus_{k\in\mathbb Z}\g^k$ equipped with a collection $(k\ge 1)$ of linear maps $l_k:\otimes^k\g\lon\g$ of degree $1$ with the property that, for any homogeneous elements $x_1,\cdots, x_n\in \g$, we have
\begin{itemize}\item[\rm(i)]
{\em (graded symmetry)} for every $\sigma\in S_{n}$,
\begin{eqnarray*}
l_n(x_{\sigma(1)},\cdots,x_{\sigma(n)})=\varepsilon(\sigma)l_n(x_1,\cdots,x_n),
\end{eqnarray*}
\item[\rm(ii)] {\em (generalized Jacobi identity)} for all $n\ge 1$,
\begin{eqnarray*}\label{sh-Lie}
\sum_{i=1}^{n}\sum_{\sigma\in  S(i,n-i) }\varepsilon(\sigma)l_{n-i+1}(l_i(x_{\sigma(1)},\cdots,x_{\sigma(i)}),x_{\sigma(i+1)},\cdots,x_{\sigma(n)})=0,
\end{eqnarray*}

\end{itemize}where $\varepsilon(\sigma)=\varepsilon(\sigma;x_1,\cdots, x_n)$ is the   Koszul sign for a permutation $\sigma\in S_n$ and $x_1,\cdots, x_n\in \g$.
\end{defi}
We denote an $L_\infty$-algebra by $(\g,\{l_k\}_{k=1}^{+\infty})$.

\begin{defi}
A {\bf Maurer-Cartan element} of an $L_\infty$-algebra $(\g, \{l_k\}_{k=1}^{+\infty})$ is an element $\alpha\in {\g}^0$ satisfying
\begin{equation}\label{L-MC}
 \sum_{k=1}^{+\infty}\frac{1}{k!}l_k(\alpha, \cdots, \alpha)=0.
\end{equation}
\end{defi}
\begin{rmk}
In general, the Maurer-Cartan equation \eqref{L-MC} makes sense when the $L_\infty$-algebra is a filtered  $L_\infty$-algebra \cite{VAD}. In the following, the $L_\infty$-algebra under consideration satisfies $l_k=0$ for $k$ sufficiently big, so the Maurer-Cartan equation make sense.
\end{rmk}
Let $\alpha$ be a Maurer-Cartan element of a $L_\infty$-algebra $(\g, \{l_k\}_{k=1}^{+\infty})$. Define $l_{k}^{\alpha}:\otimes^{k} \g\lon \g ~~(k\geq1)$ by
\begin{equation*}
l_{k}^{\alpha}(x_1,\cdots,x_k)=\sum_{n=0}^{+\infty}\frac{1}{n!}l_{k+n}(\underbrace{\alpha,\cdots,\alpha}_{n},x_1,\cdots, x_k).
\end{equation*}

\begin{thm}\label{twistLin}{\rm(\cite{DSV,Get})}
With the above notation,  $(\g,\{l_k^{\alpha}\}_{k=1}^{+\infty})$ is an $L_\infty$-algebra which is called the twisted $L_\infty$-algebra by $\alpha$.
\end{thm}
\begin{thm}\label{twisting-deformation}{\rm(\cite{Get})}
Let $(\g,\{l_k^{\alpha}\}_{k=1}^{+\infty})$ be a twist $L_\infty$-algebra by $\alpha$. Then $\alpha+\alpha'$ is a Maurer-Cartan element of the $L_\infty$-algebra $(\g,\{l_k\}_{k=1}^{+\infty})$ if and only if $\alpha'$ is a Maurer-Cartan element of the twist $L_\infty$-algebra $(\g,\{l_k^{\alpha}\}_{k=1}^{+\infty})$.
\end{thm}

\begin{defi}\rm(\cite{Vo})
A {\bf $V$-data} consists of a quadruple $(L,\h,P,\Delta)$, where
\begin{itemize}
\item[$\bullet$] $(L,[\cdot,\cdot])$ is a graded Lie algebra,
\item[$\bullet$] $\h$ is an abelian graded Lie subalgebra of $(L,[\cdot,\cdot])$,
\item[$\bullet$] $P:L\lon L$ is a projection, that is $P\circ P=P$, whose image is $\h$ and its kernel is a  graded Lie subalgebra of $(L,[\cdot,\cdot])$,
\item[$\bullet$] $\Delta$ is an element in $\ker(P)^{1}$ such that $[\Delta,\Delta]=0$.
\end{itemize}
\end{defi}

\begin{thm}\rm(\cite{Vo,YF})\label{direct-sum}
Let $(L,\h,P,\Delta)$ be a $V$-data. Then the graded vector space $s^{-1}L\oplus \h$ is an $L_\infty$-algebra, where
\begin{eqnarray*}
l_1(s^{-1}f,\theta)&=&(-s^{-1}[\Delta,f],P(f+[\Delta,\theta])),\\
l_2(s^{-1}f,s^{-1}g)&=&(-1)^{|f|}s^{-1}[f,g],\\
l_k(s^{-1}f,\theta_1,\dots,\theta_{k-1})&=&P([\cdots[[f, \theta_1], \theta_2], \cdots, \theta_{k-1}]),\quad k\geq 2,\\
l_k(\theta_1,\dots,\theta_{k-1},\theta_k)&=&P([\cdots[[\Delta, \theta_1], \theta_2], \cdots, \theta_k]),\quad k\geq 2.
\end{eqnarray*}
Here $\theta,\theta_1,\dots,\theta_k$ are homogeneous elements of $\h$ and $f,g$ are homogeneous elements of $L$. All the other $L_\infty$-algebra products that are not obtained from the ones written above by permutations of arguments, will vanish.

Moreover, if $L'$ is a graded Lie subalgebra of $L$ that satisfies $[\Delta,L']\subset L'$, then $s^{-1}L'\oplus \h$ is an $L_\infty$-subalgebra of the above $L_\infty$-algebra $(s^{-1}L\oplus \h,\{l_k\}_{k=1}^{+\infty})$.
\end{thm}

\subsection{The $L_{\infty}$-algebra that controls deformations of pre-LieDer pairs}
\begin{pro}
We have a $V$-data $(L,\h,P,\Delta)$ as follows:
\begin{itemize}
\item[$\bullet$] the graded Lie algebra $(L,[\cdot,\cdot])$ is given by $(\oplus_{n=1}^{+\infty}\Hom(\wedge^{n-1}(\g\oplus V)\otimes(\g\oplus V),\g\oplus V),[\cdot,\cdot]^{MN})$;
\item[$\bullet$] The abelian graded Lie subalgebra $\h$ is given by $$\h=\oplus_{n=0}^{+\infty}C^{n|-1}(\g\oplus V,\g\oplus V)=\oplus_{n=0}^{+\infty}\Hom(\wedge^{n-1}\g\otimes \g, V);$$
\item[$\bullet$] $P:L\longrightarrow L$ is a projection onto the subspace $\h$;
\item[$\bullet$] $\Delta=0$.
\end{itemize}
Consequently, we obtain an $L_{\infty}$-algebra $(s^{-1}L\oplus \h,\{l_k\}_{k=1}^{+\infty})$, where $l_i$ are given by
\begin{eqnarray*}
l_1(s^{-1}f,\theta)&=&P(f),\\
l_2(s^{-1}f,s^{-1}g)&=&(-1)^{|f|}s^{-1}[f,g]^{MN},\\
l_k(s^{-1}f,\theta_1,\dots,\theta_{k-1})&=&P([\cdots[[f, \theta_1]^{MN}, \theta_2]^{MN}, \cdots, \theta_{k-1}]^{MN}),\quad k\geq 2,
\end{eqnarray*}
for homogeneous elements $\theta,\theta_1,\dots,\theta_{k-1}\in\h$, $f,g\in L$ and all the other possible combinations vanish.
\end{pro}
\begin{proof}
Obviously, $\h$ is an abelian graded Lie subalgebra of $(L,[\cdot,\cdot])$ and $\Ker(P)$ is a graded Lie subalgebra of $(L,[\cdot,\cdot])$. Thus, $(L,\h,P,\Delta)$ is a $V$-data. By Theorem \ref{direct-sum}, we obtain an $L_{\infty}$-algebra $(s^{-1}L\oplus \h,\{l_k\}_{k=1}^{+\infty})$.
\end{proof}
We set $L'$ by
\begin{equation*}
L'=\oplus_{n=0}^{+\infty}C^{n|0}(\g\oplus V,\g\oplus V)= \oplus_{n=0}^{+\infty}(\Hom(\wedge^n\g\otimes\g,\g)\oplus\Hom(\wedge^n\g\otimes V,V)\oplus\Hom(\wedge^{n-1}\g\otimes V\otimes\g,V)).
\end{equation*}
Now, we give the main result in this subsection.
\begin{thm}\label{L-algebra}
With the above notation, $(s^{-1}L'\oplus \h,\{l_k\}_{k=1}^{+\infty})$ is an $L_{\infty}$-algebra, where $l_k$ are given by
\begin{eqnarray*}
l_2(s^{-1}f,s^{-1}g)&=&(-1)^{|f|}s^{-1}[f,g]^{MN},\\
l_2(s^{-1}f,\theta)&=&P([f, \theta]^{MN}),\\
l_k(s^{-1}f,\theta_1,\dots,\theta_{k-1})&=&0,\quad k\geq 3,
\end{eqnarray*}
for homogeneous elements $\theta,\theta_1,\dots,\theta_{k-1}\in\h$, $f,g\in L'$ and all the other possible combinations vanish.

Moreover, for all $\pi\in\Hom(\g\otimes \g,V)$, $\rho\in\Hom(\g\otimes V,V)$, $\mu\in\Hom(V\otimes \g,V)$ and $D\in\Hom(\g,V)$, $(s^{-1}(\pi+\rho+\mu),D)$ is a Maurer-Cartan element of the $L_\infty$-algebra $(s^{-1}L'\oplus \h,\{l_k\}_{k=1}^{+\infty})$ if and only if $(\g,D,\rho,\mu)$ is a pre-LieDer pair.
\end{thm}
\begin{proof}
Obviously, $L'$ is a graded Lie subalgebra. Since $\Delta=0$, we have $[\Delta,L']$=0. Thus, $(s^{-1}L'\oplus \h,\{l_k\}_{k=1}^{+\infty})$ is an $L_{\infty}$-algebra. It is straightforward to deduce that
\begin{equation*}
[\pi+\rho+\mu,D]^{MN}\in \Hom(\g\otimes \g,V), \quad [[\pi+\rho+\mu,D]^{MN},D]^{MN}=0.
\end{equation*}
Then, we have
\begin{eqnarray*}
&&\frac{1}{2}l_2((s^{-1}(\pi+\rho+\mu),D), (s^{-1}(\pi+\rho+\mu),D))\\
&=&\frac{1}{2}l_2(s^{-1}(\pi+\rho+\mu),s^{-1}(\pi+\rho+\mu))+l_2(s^{-1}(\pi+\rho+\mu),D)\\
&=&(-\frac{1}{2}s^{-1}[\pi+\rho+\mu,\pi+\rho+\mu]^{MN},P[\pi+\rho+\mu,D]^{MN})\\
&=&(-\frac{1}{2}s^{-1}[\pi+\rho+\mu,\pi+\rho+\mu]^{MN},[\pi+\rho+\mu,D]^{MN}).
\end{eqnarray*}
Thus, $(s^{-1}(\pi+\rho+\mu),D)$ is a Maurer-Cartan element of the $L_\infty$-algebra $(s^{-1}L'\oplus \h,\{l_k\}_{k=1}^{+\infty})$ if and only if
\begin{equation*}
[\pi+\rho+\mu,\pi+\rho+\mu]^{MN}=0, \quad [\pi+\rho+\mu,D]^{MN}=0.
\end{equation*}
By Proposition \ref{MC-prelierep}, $[\pi+\rho+\mu,\pi+\rho+\mu]^{MN}=0$ if and only if $(V;\rho,\mu)$ is a representation of the pre-Lie algebra $(\g,\pi)$. For all $x,y\in \g$, we have
\begin{eqnarray*}
0&=&[\pi+\rho+\mu,D]^{MN}(x,y)\\
&=&\rho(x,D(y))+\mu(D(x),y)-D(\pi(x,y))\\
&=&\rho(x)D(y)+\mu(y)D(x)-D(x\cdot y),
\end{eqnarray*}
which implies that $[\pi+\rho+\mu,D]^{MN}=0$ if and only if $D$ is a derivation. Thus, $(s^{-1}(\pi+\rho+\mu),D)$ is a Maurer-Cartan element of the $L_\infty$-algebra $(s^{-1}L'\oplus \h,\{l_k\}_{k=1}^{+\infty})$ if and only if $(\g,D,\rho,\mu)$ is a pre-LieDer pair. This finishes the proof.
\end{proof}
\begin{thm}
Let $(\g,D,\rho,\mu)$ be a pre-LieDer pair. Then $(s^{-1}L'\oplus \h,\{l_k^{(s^{-1}(\pi+\rho+\mu),D)}\}_{k=1}^{+\infty})$ is an $L_\infty$-algebra, where $l_k^{(s^{-1}(\pi+\rho+\mu),D)}$ are given by
\begin{eqnarray*}
l_1^{(s^{-1}(\pi+\rho+\mu),D)}(s^{-1}f,\theta)&=&l_2((s^{-1}(\pi+\rho+\mu),D),(s^{-1}f,\theta)),\\
l_2^{(s^{-1}(\pi+\rho+\mu),D)}((s^{-1}f_1,\theta_1),(s^{-1}f_2,\theta_2))&=&l_2((s^{-1}f_1,\theta_1),(s^{-1}f_2,\theta_2)),\\
l_k^{(s^{-1}(\pi+\rho+\mu),D)}((s^{-1}f_1,\theta_1),\dots,(s^{-1}f_k\theta_k))&=&0,\quad k\geq 3,
\end{eqnarray*}

Furthermore, for all $\pi'\in\Hom(\g\otimes \g,V)$, $\rho'\in\Hom(\g\otimes V,V)$, $\mu'\in\Hom(V\otimes \g,V)$ and $D'\in\Hom(\g,V)$, $(\g,\pi+\pi',D+D',\rho+\rho',\mu+\mu')$ is a pre-LieDer pair if and only if $(s^{-1}(\pi'+\rho'+\mu'),D')$ is a Maurer-Cartan element of the twist $L_\infty$-algebra $(s^{-1}L'\oplus \h,\{l_k^{(s^{-1}(\pi+\rho+\mu),D)}\}_{k=1}^{+\infty})$.
\end{thm}
\begin{proof}
By Theorem \ref{L-algebra}, $(\g,\pi+\pi',D+D',\rho+\rho',\mu+\mu')$ is a pre-LieDer pair if and only if $(s^{-1}(\pi+\pi'+\rho+\rho'+\mu+\mu'),D+D')$ is a Maurer-Cartan element of the $L_\infty$-algebra $(s^{-1}L'\oplus \h,\{l_k\}_{k=1}^{+\infty})$. By Theorem \ref{twisting-deformation}, $(s^{-1}(\pi+\pi'+\rho+\rho'+\mu+\mu'),D+D')$ is a Maurer-Cartan element of the $L_\infty$-algebra $(s^{-1}L'\oplus \h,\{l_k\}_{k=1}^{+\infty})$ if and only if $(s^{-1}(\pi'+\rho'+\mu'),D')$ is a Maurer-Cartan element of the twist $L_\infty$-algebra $(s^{-1}L'\oplus \h,\{l_k^{(s^{-1}(\pi+\rho+\mu),D)}\}_{k=1}^{+\infty})$.
\end{proof}
\section{Cohomologies and infinitesimal deformations of pre-LieDer pairs}\label{sec:infinitesimal-deformation}
In this section, first, we give the cohomology of pre-LieDer pairs. Then, we study infinitesimal deformations of pre-LieDer pairs by using this cohomology . We show that equivalent infinitesimal deformations are in the same second cohomology group. Finally, we use this cohomology to give the cohomology of  regular pre-LieDer pairs.

In this subsection, we will also denote the pre-Lie multiplication $\cdot$ by $\pi$.
\subsection{Cohomologies of pre-LieDer pairs}
\
\newline
\indent
Let $(\g,\pi,D,\rho,\mu)$ be a pre-LieDer pair. Define the set of $0$-cochains $C^0(\g,\pi,\rho,\mu,D)$ to be $0$ and define the set of $1$-cochains $C^1(\g,\pi,\rho,\mu,D)$ to be $\Hom(\g,\g)\oplus \Hom(V,V)$. For $n\geq 2$, define the set of $n$-cochains $C^n(\g,\pi,\rho,\mu,D)$ by
\begin{eqnarray*}
&&C^n(\g,\pi,\rho,\mu,D):=C^n(\g,\pi,\rho,\mu)\oplus C^{n-1}(\g;V)\\
&=&\Hom(\wedge^{n-1}\g\otimes \g,\g)\oplus \Hom(\wedge^{n-1}\g\otimes V,V)\oplus \Hom(\wedge^{n-2}\g\otimes V\otimes \g,V)\oplus \Hom(\wedge^{n-2}\g\otimes \g,V).
\end{eqnarray*}
Define the coboundary operator $\huaD:C^n(\g,\pi,\rho,\mu,D)\longrightarrow C^{n+1}(\g,\pi,\rho,\mu,D)$ by
\begin{equation}\label{cohomology-operator}
\huaD (f,\theta)=(-1)^{n-2}l_1^{(s^{-1}(\pi+\rho+\mu),D)}(s^{-1}f,\theta),
\end{equation}
where $f\in \Hom(\wedge^{n-1}\g\otimes \g,\g)\oplus \Hom(\wedge^{n-1}\g\otimes V,V)\oplus \Hom(\wedge^{n-2}\g\otimes V\otimes \g,V)$ and $\theta\in  \Hom(\wedge^{n-2}\g\otimes \g,V)$.
\begin{thm}
$(\oplus_{n=1}^{+\infty}C^n(\g,\pi,\rho,\mu,D),\huaD)$ is a cochain complex, i.e. $\huaD\circ \huaD=0$.
\end{thm}
\begin{proof}
Since $(s^{-1}L'\oplus \h,\{l_k^{(s^{-1}(\pi+\rho+\mu),D)}\}_{k=1}^{+\infty})$ is an $L_\infty$-algebra, we have $l_1^{(s^{-1}(\pi+\rho+\mu),D)}\circ l_1^{(s^{-1}(\pi+\rho+\mu),D)}=0$, which implies that $\huaD\circ \huaD=0$.
\end{proof}
\begin{defi}
The cohomology of the cochain complex $(\oplus_{n=1}^{+\infty}C^n(\g,\pi,\rho,\mu,D),\huaD)$ is called the cohomology of the pre-LieDer pair $(\g,D,\rho,\mu)$. The corresponding cohomology group is denoted by $H^n(\g,\pi,\rho,\mu,D)$.
\end{defi}
In the sequel, we give the explicit formula of the operator $\huaD$. For all $f\in \Hom(\wedge^{n-1}\g\otimes \g,\g)\oplus \Hom(\wedge^{n-1}\g\otimes V,V)\oplus \Hom(\wedge^{n-2}\g\otimes V\otimes \g,V)$ and $\theta\in  \Hom(\wedge^{n-2}\g\otimes \g,V)$, we have
\begin{equation*}
||[\pi+\rho+\mu,f]^{MN}||=n|0, \quad ||[\pi+\rho+\mu,\theta]^{MN}||=n|-1,\quad ||[f, D]^{MN}||=n|-1.
\end{equation*}

Thus, by \eqref{cohomology-operator},we have
\begin{eqnarray*}
\huaD(f,\theta)&=&(-1)^{n-2}l_1^{(s^{-1}(\pi+\rho+\mu),D)}(s^{-1}f,\theta)\\
&=&(-1)^{n-2}l_2((s^{-1}(\pi+\rho+\mu),D),(s^{-1}f,\theta))\\
&=&(-1)^{n-2}(l_2(s^{-1}(\pi+\rho+\mu),s^{-1}f)+l_2(s^{-1}(\pi+\rho+\mu),\theta)+l_2(D,s^{-1}f))\\
&=&(-1)^{n-2}(-s^{-1}[\pi+\rho+\mu,f]^{MN},P[\pi+\rho+\mu,\theta]^{MN}+P[f,D]^{MN})\\
&=&(-1)^{n-2}(-s^{-1}[\pi+\rho+\mu,f]^{MN},[\pi+\rho+\mu,\theta]^{MN}+[f,D]^{MN})\\
&=&(\partial f,\dM \theta+\delta f).
\end{eqnarray*}
\begin{lem}
With the above notation, for all $f=(f_{\g},f_{\rho},f_{\mu})\in \Hom(\wedge^{n-1}\g\otimes \g,\g)\oplus \Hom(\wedge^{n-1}\g\otimes V,V)\oplus \Hom(\wedge^{n-2}\g\otimes V\otimes \g,V)$, $\delta:\Hom(\wedge^{n-1}\g\otimes \g,\g)\oplus \Hom(\wedge^{n-1}\g\otimes V,V)\oplus \Hom(\wedge^{n-2}\g\otimes V\otimes \g,V)\longrightarrow \Hom(\wedge^{n-1}\g\otimes \g,V)$ is given by
\begin{eqnarray*}
(\delta f)(x_1,\dots,x_{n-1}x_n)&=&\sum_{i=1}^{n-1}(-1)^{i+1}f_{\mu}(x_1,\dots,\hat{x_i},\dots,x_{n-1},D(x_i),x_n)\\
&&+(-1)^{n-2}(f_{\rho}(x_1,\dots,x_{n-1},D(x_n))-D(f_{\g}(x_1,\dots,x_n))).
\end{eqnarray*}
\end{lem}
\begin{proof}
For all $x_1,\dots,x_n\in \g$, we have
\begin{eqnarray*}
&&(\delta f)(x_1,\dots,x_n)\\
&=&(-1)^{n-2}[f,D]^{MN}(x_1,\dots,x_n)\\
&=&(-1)^{n-2}(\sum_{i=1}^{n-1}(-1)^{i+1}f_{\mu}(D(x_i),x_1,\dots,\hat{x_i},\dots,x_n)+f_{\rho}(x_1,\dots,x_{n-1},D(x_n))-D(f_{\g}(x_1,\dots,x_n)))\\
&=&\sum_{i=1}^{n-1}(-1)^{i+1}f_{\mu}(x_1,\dots,\hat{x_i},\dots,x_{n-1},D(x_i),x_n)+(-1)^{n-2}(f_{\rho}(x_1,\dots,x_{n-1},D(x_n))-D(f_{\g}(x_1,\dots,x_n))).
\end{eqnarray*}
The proof is finished.
\end{proof}
The following diagram well explain the above operators:
\begin{center}
   \begin{tikzpicture}
      $$\matrix (m) [matrix of math nodes,row sep=3em,column sep=5em,minimum width=2em] {
			 C^n(\g,\pi,\rho,\mu) &  C^{n+1}(\g,\pi,\rho,\mu) & C^{n+2}(\g,\pi,\rho,\mu)  \\
		\huaC^{n-1}(\g;V) &  \huaC^n(\g;V) & \huaC^{n+1}(\g;V)\\
		};
		\path[->,auto] (m-2-1) edge node[swap]{$\dM$}                  (m-2-2);
        \path[->,auto] (m-2-2) edge node[swap]{$\dM$}                  (m-2-3);
		\path[->,auto] (m-1-1) edge node {$\partial$}                  (m-1-2);
        \path[->,auto] (m-1-2) edge node {$\partial$}                  (m-1-3);
        \path[->,below] (m-1-1) edge node[above] {$\delta$}                  (m-2-2);
        \path[->,below] (m-1-2) edge node[above] {$\delta$}                  (m-2-3);
		\path   (m-2-1) edge[draw=none]   node [sloped] {$\oplus$} (m-1-1);
		\path   (m-2-2) edge[draw=none]   node [sloped] {$\oplus$} (m-1-2);
        \path   (m-2-3) edge[draw=none]   node [sloped] {$\oplus$} (m-1-3);$$
	\end{tikzpicture}
\end{center}
At the end of this subsection, we give the relation between various cohomologies groups.
\begin{thm}
There is a short exact sequence of the cochain complexes:
$$
0\longrightarrow(\oplus_{n=0}^{+\infty}C^n(\g;V),\dM)\stackrel{\iota}{\longrightarrow}(\oplus_{n=0}^{+\infty}C^n(\g,\pi,\rho,\mu,D),\huaD)\stackrel{p}{\longrightarrow} (\oplus_{n=0}^{+\infty}C^n(\g,\pi,\rho,\mu),\partial)\longrightarrow 0,
$$
where $\iota(\theta)=(0,\theta)$ and $p(f,\theta)=f$ for all $f\in C^n(\g,\pi,\rho,\mu)$ and $\theta\in C^{n-1}(\g;V)$.

Consequently, there is a long exact sequence of the cohomology groups:
$$
\cdots\longrightarrow H^n(\g;V)\stackrel{H^n(\iota)}{\longrightarrow} H^n(\g,\pi,\rho,\mu,D)\stackrel{H^n(p)}{\longrightarrow} H^n(\g,\pi,\rho,\mu)\stackrel{c^n}{\longrightarrow}H^{n+1}(\g;V)\longrightarrow \cdots,
$$
where the connecting map $c^n$ is defined by $c^n([\alpha])=[\delta(\alpha)]$, for all $[\alpha]\in H^n(\g,\pi,\rho,\mu).$
\end{thm}
\begin{proof}
By the explicit formula of the coboundary operator $\huaD$, we have the short exact sequence of chain complexs which induces a long exact sequence of cohomology groups.
\end{proof}

\subsection{Infinitesimal deformations of pre-LieDer pairs}
\begin{defi}
Let $(\g,\pi,D,\rho,\mu)$ be a pre-LieDer pair. If for all $t\in \mathbb{K}$, $(\g,\pi_t=\pi+t\omega,D_t=D+t\hat{D},\rho_t=\rho+t\sigma,\mu_t=\mu+t\tau)$ is a pre-LieDer pair, where $\omega:\g\otimes \g\longrightarrow \g,\sigma:\g\otimes V\longrightarrow\g, \tau:V\otimes \g\longrightarrow \g$ and $\hat{D}:\g\longrightarrow V$ are linear maps. We say that $(\omega,\sigma,\tau,\hat{D})$ generates an {\bf infinitesimal deformation} of a pre-LieDer pair $(\g,\pi,D,\rho,\mu)$.
\end{defi}
It is direct to check that $(\omega,\sigma,\tau,\hat{D})$ generates an infinitesimal deformation of a pre-LieDer pair $(\g,\pi,D,\rho,\mu)$ if and only if for all $x,y\in \g$, the following equalities are satisfied:
\begin{eqnarray}
\label{deformation-1}[\pi+\rho+\mu,\omega+\sigma+\tau]^{MN}&=&0,\\
\label{deformation-2}[\omega+\sigma+\tau,\omega+\sigma+\tau]^{MN}&=&0,\\
\label{deformation-3}[\pi+\rho+\mu,\hat{D}]^{MN}+[\omega+\sigma+\tau,D]^{MN}&=&0,\\
\label{deformation-4}[\omega+\sigma+\tau,\hat{D}]^{MN}&=&0.
\end{eqnarray}
Obviously, $(\omega,\sigma,\tau) \in C^2(\g,\pi,\rho,\mu)$, \eqref{deformation-1} means that $(\omega,\sigma,\tau)$ is a $2$-cocycle of the pre-Lie algebra with representations, that is $\partial(\omega,\sigma,\tau)=0$. \eqref{deformation-2} means that $(V;\sigma,\tau)$ is a representation of the pre-Lie algebra $(\g,\omega)$. \eqref{deformation-3} means that $\dM \hat{D}+\delta(\omega,\sigma,\tau)=0$. \eqref{deformation-4} means that $(\g,\omega,\hat{D},\sigma,\tau)$ is a pre-LieDer pair. By \eqref{deformation-1} and \eqref{deformation-3}, we have $$\huaD(\omega,\sigma,\tau,\hat{D})=(\partial(\omega,\sigma,\tau),\dM \hat{D}+\delta(\omega,\sigma,\tau))=0.$$
\begin{thm}\label{2-cocycle}
With the above notation, $(\omega,\sigma,\tau,\hat{D})$ is a $2$-cocycle of the pre-LieDer pair $(\g,D,\rho,\mu)$.
\end{thm}

\begin{defi}
Let $(\pi'_t=\pi+t\omega',\rho'_t=\rho+t\sigma',\mu_t'=\mu+t\tau',D_t'=D+t\hat{D}')$ and $(\pi_t=\pi+t\omega,\rho_t=\rho+t\sigma,\mu_t=\mu+t\tau,D_t=D+t\hat{D})$ be two infinitesimal deformations of the pre-LieDer pair $(\g,\pi,D,\rho,\mu)$. We call them {\bf equivalent} if there exists $N\in\gl(\g)$ and $S\in \gl(V)$ such that $({\Id}_{\g}+tN,{\Id}_V+tS)$ is a homomorphism from the pre-LieDer pair $(\g,\pi'_t,D_t',\rho'_t,\mu_t')$ to the pre-LieDer pair $(\g,\pi_t,D_t,\rho_t,\mu_t)$, i.e. for all $x,y\in \g$, the following equalities hold:
\begin{eqnarray*}
({\Id}_{\g}+tN)(x\cdot_t' y)&=&({\Id}_{\g}+tN)(x)\cdot_t({\Id}_{\g}+tN)(y),\\
({\Id}_V+tS)\circ\rho'_t(x)&=&\rho_t(({\Id}_{\g}+tN)(x))\circ ({\Id}_V+tS),\\
({\Id}_V+tS)\circ\mu'_t(x)&=&\mu_t(({\Id}_{\g}+tN)(x))\circ ({\Id}_V+tS),\\
({\Id}_V+tS)\circ D_t'&=&D_t\circ ({\Id}_{\g}+tN).
\end{eqnarray*}
An infinitesimal deformation is said to be {\bf trivial} if it equivalent to $(\pi,\rho,\mu,D)$.
\end{defi}
By direct calculations, $(\pi_t',\rho_t',\mu_t',D_t')$ and $(\pi_t,\rho_t,\mu_t,D_t)$ are equivalent deformations if and only if for all $x,y\in \g, u\in V$, the following equalities hold:
\begin{eqnarray}
\label{equi-deformation-1}\omega'(x,y)-\omega(x,y)&=&N(x)\cdot y+x\cdot N(y)-N(x\cdot y),\\
\label{equi-deformation-2}N(\omega'(x,y))&=&N(x)\cdot N(y)+\omega(x,N(y))+\omega(N(x),y),\\
\label{equi-deformation-3}\omega(N(x),N(y))&=&0,\\
\label{equi-deformation-4}\sigma'(x)u-\sigma(x)u&=&\rho(x)S(u)+\rho(N(x))u-S(\rho(x)u),\\
\label{equi-deformation-5}S(\sigma'(x)u)&=&\rho(N(x))S(u)+\sigma(x)S(u)+\sigma(N(x))u,\\
\label{equi-deformation-6}\sigma(N(x))S(u)&=&0\\
\label{equi-deformation-7}\tau'(x)u-\tau(x)u&=&\mu(x)S(u)+\mu(N(x))u-S(\mu(x)u),\\
\label{equi-deformation-8}S(\tau'(x)u)&=&\mu(N(x))S(u)+\tau(x)S(u)+\tau(N(x))u,\\
\label{equi-deformation-9}\tau(N(x))S(u)&=&0,\\
\label{equi-deformation-10}\hat{D}'(x)-\hat{D}(x)&=&D(N(x))-S(D(x)),\\
\label{equi-deformation-11}S(\hat{D}'(x))&=&\hat{D}(N(x)).
\end{eqnarray}
We summarize the above discussion into the following conclusion:
\begin{thm}
Let $(\g,D,\rho,\mu)$ be a pre-LieDer pair. If two infinitesimal deformations $(\pi'_t=\pi+t\omega',\rho'_t=\rho+t\sigma',\mu_t'=\mu+t\tau',D_t'=D+t\hat{D}')$ and $(\pi_t=\pi+t\omega,\rho_t=\rho+t\sigma,\mu_t=\mu+t\tau,D_t=D+t\hat{D})$ are equivalent, then $(\omega',\sigma',\tau',\hat{D}')$ and $(\omega,\sigma,\tau,\hat{D})$ are in the same cohomology class of $H^2(\g,\pi,\rho,\mu,D)$.
\end{thm}
\begin{proof}
By Theorem \ref{2-cocycle}, we obtain that $(\omega',\sigma',\tau',\hat{D}'), (\omega,\sigma,\tau,\hat{D})\in Z^2(\g,\pi,\rho,\mu,D)$. Obviously, $(N,S)\in C^1(\g,\pi,\rho,\mu,D)$. For all $x,y\in \g, u,v\in V$, by \eqref{equi-deformation-1}, \eqref{equi-deformation-4} and \eqref{equi-deformation-7}, we have
\begin{eqnarray*}
&&((\omega',\sigma',\tau')-(\omega,\sigma,\tau))((x,u),(y,v))\\
&=&(N(x)\cdot y+x\cdot N(y)-N(x\cdot y),\rho(x)S(v)+\rho(N(x))v-S(\rho(x)v),\\
&&\mu(y)S(u)+\mu(N(y))u-S(\mu(y)u))\\
&=&\partial(N,S)((x,u),(y,v)),
\end{eqnarray*}
which implies that $(\omega',\sigma',\tau')-(\omega,\sigma,\tau)=\partial(N,S)$.
For all $x\in \g$, by \eqref{equi-deformation-10}, we have
\begin{equation*}
(\hat{D}'-\hat{D})(x)=D(N(x))-S(D(x))=\delta(N,S)(x),
\end{equation*}
which implies that $\hat{D}'-\hat{D}=\delta(N,S)$.
Thus, we have
\begin{equation*}
(\omega',\sigma',\tau',\hat{D}')-(\omega,\sigma,\tau,\hat{D})=(\partial(N,S),\delta(N,S))=\huaD(N,S).
\end{equation*}
This finishes the proof.
\end{proof}
\subsection{Cohomologies of regular pre-LieDer pairs}
\
\newline
\indent
Let $(\g,D)$ be a regular pre-LieDer pair. Define the set of $0$-cochains $C^0(\g,D)$ to be $0$ and define the set of $1$-cochains $C^1(\g,D)$ to be $\Hom(\g,\g)$. For $n\geq 2$, define the set of $n$-cochains $C^n(\g,D)$ by
\begin{equation*}
 C^n(\g,D):=\Hom(\wedge^{n-1}\g\otimes \g,\g)\oplus \Hom(\wedge^{n-2}\g\otimes\g,\g).
\end{equation*}
For all $f\in \Hom(\wedge^{n-1}\g\otimes \g,\g)$ and $\theta\in \Hom(\wedge^{n-2}\g\otimes \g,\g)$. Define the embedding $i:C^n(\g,D)\longrightarrow C^n(\g,\pi,L,R,D)$ by
\begin{equation*}
i(f,\theta)=(f,f,f,\theta).
\end{equation*}
Denote by $\Img^n(i)=i(C^n(\g,D))$. Then we have
\begin{pro}
With the above notation, $(\oplus_{n=0}^{+\infty}\Img^n(i),\huaD)$ is a subcomplex of the cochain complex $(\oplus_{n=0}^{+\infty}C^n(\g,\pi,L,R,D),\huaD)$.
\end{pro}
\begin{proof}
For all $f\in \Hom(\wedge^{n-1}\g\otimes \g,\g)$, $\theta\in \Hom(\wedge^{n-2}\g\otimes \g,\g)$, $(f,f,f,\theta)\in \Img^n(i)$, we have
\begin{eqnarray*}
\huaD(f,f,f,\theta)&=&(\partial(f,f,f),\dM \theta+\delta(f,f,f))\\
&=&(\partial(f,f,f)_{\g},\partial(f,f,f)_{\rho},\partial(f,f,f)_{\mu},\dM_{\reg} \theta+\delta(f,f,f)).
\end{eqnarray*}
By \eqref{cohomology-1}, we have $\partial(f,f,f)_{\g}=\dM_{\reg} f$. For all $x_1,\dots,x_{n+1}\in \g$, we have
\begin{eqnarray*}
\partial (f,f,f)_{\rho}(x_1,\dots,x_{n+1})&=&\sum_{i=1}^n(-1)^{i+1}f(x_1,\dots,\hat{x_i},\dots,x_n,x_i)\cdot x_{n+1}\\
&&+\sum_{i=1}^n(-1)^{i+1}x_i\cdot f(x_1,\dots,\hat{x_i},\dots,x_{n+1})\\
&&-\sum_{i=1}^n(-1)^{i+1}f(x_1,\dots,\hat{x_i}\dots,x_n,x_i\cdot x_{n+1})\\
&&+\sum_{1\leq i<j\leq n}(-1)^{i+j} f([x_i,x_j]_C,x_1,\dots,\hat{x_i},\dots,\hat{x_j},\dots,x_{n+1})\\
&=&\dM_{\reg} f(x_1,\dots,x_{n+1}).
\end{eqnarray*}
and
\begin{eqnarray*}
&&\partial (f,f,f)_{\mu}(x_1,\dots,x_{n+1})\\
&=&(-1)^{n-1}(x_n\cdot f(x_1,\dots,x_{n-1},x_{n+1})+ f(x_1,\dots,x_n)\cdot x_{n+1}-f(x_1,\dots,x_{n-1},x_n\cdot x_{n+1})\\
&&-\sum_{i=1}^{n-1}(-1)^{i+1}f(x_1,\dots,\hat{x_i},\dots,x_n,x_i\cdot x_{n+1})+\sum_{i=1}^{n-1}(-1)^{i+1}x_i\cdot f(x_1,\dots,\hat{x_i},\dots,x_{n+1})\\
&&-\sum_{i=1}^{n-1}(-1)^{i+1}f(x_1,\dots,\hat{x_i}\dots,x_{n-1},x_i\cdot x_n,x_{n+1})+\sum_{i=1}^{n-1}(-1)^{i+1}f(x_1,\dots,\hat{x_i},\dots,x_n,x_i)\cdot x_{n+1}\\
&&+\sum_{i=1}^{n-1}(-1)^{i+1}f(x_1,\dots,\hat{x_i},\dots,x_{n-1},x_n\cdot x_i,x_{n+1})\\
&&+\sum_{1\leq i<j\leq n-1}(-1)^{i+j} f([x_i,x_j]_C,x_1,\dots,\hat{x_i},\dots,\hat{x_j},\dots,x_{n+1})\\
&=&\dM_{\reg} f(x_1,\dots,x_{n+1}).
\end{eqnarray*}
Thus, we obtain that $$\huaD(f,f,f,\theta)=(\dM_{\reg} f,\dM_{\reg} f,\dM_{\reg} f,\dM_{\reg} \theta+\delta(f,f,f))=i(\dM_{\reg} f,\dM_{\reg} \theta+\delta(f,f,f)),$$
which implies that $(\oplus_{n=0}^{+\infty}\Img^n(i),\huaD)$ is a subcomplex.
\end{proof}
Define the projection $p:\Img^n(i)\longrightarrow C^n(\g,D)$ by
\begin{equation*}
p(f,f,f,\theta)=(f,\theta), \quad \forall f\in \Hom(\wedge^{n-1}\g\otimes \g,\g), \theta\in \Hom(\wedge^{n-2}\g\otimes \g,\g).
\end{equation*}
Then for $n\geq 1$, we define $\bar{\huaD}:C^n(\g,D)\longrightarrow C^{n+1}(\g,D)$ by $\bar{\huaD}=p\circ \huaD \circ i$.
More precisely, we have
\begin{equation*}
\bar{\huaD}(f,\theta)=(\dM_{\reg} f,\dM_{\reg} \theta+\Omega f), \quad \forall f\in \Hom(\wedge^{n-1}\g\otimes \g,\g), \theta\in \Hom(\wedge^{n-2}\g\otimes \g,\g),
\end{equation*}
where $\Omega:\Hom(\wedge^{n-1}\g\otimes \g,\g)\longrightarrow \Hom(\wedge^{n-1}\g\otimes \g,\g)$ is given by
\begin{equation*}
\Omega(f)(x_1,\dots,x_n)=(-1)^{n-2}(\Sigma_{i=1}^nf(x_1,\dots,x_{i-1},D(x_i),x_{i+1},\dots,x_n)-D(f(x_1,\dots,x_n))).
\end{equation*}
\begin{thm}\label{cochain complex}
With the above notation, $(\oplus_{n=0}^{+\infty}C^n(\g,D),\bar{\huaD})$ is a cochain complex, i.e. $\bar{D}\circ \bar{D}=0$.
\end{thm}
\begin{proof}
By $\bar{\huaD}=p\circ \huaD \circ i$ and $i\circ p=\Id$, we have
\begin{equation*}
\bar{D}\circ \bar{D}=p\circ D\circ i\circ p\circ D\circ i=p\circ D\circ D\circ i=0.
\end{equation*}
This finishes the proof.
\end{proof}
\begin{defi}
The cohomology of the cochain complex $(\oplus_{n=1}^{+\infty}C^n(\g,D),\bar{\huaD})$ is taken to be the cohomology of the regular pre-LieDer pair $(\g,D)$. The corresponding cohomology group is denoted by $H^n(\g,D)$.
\end{defi}
\section{Abelian extensions of regular pre-LieDer pairs}\label{sec:abelian-extension}
In this section, first, we introduce the notion of representations of regular pre-LieDer pairs. Then, we give cohomologies of regular pre-LieDer pairs with coefficients in an arbitrary representation. Finally, we study abelian extensions of  regular pre-LieDer pairs using this cohomological approach. We show that abelian extensions are classified by the second  cohomology group.
\begin{defi}
A {\bf representation} of a regular pre-LieDer pair $(\g,D)$ on a vector space $V$ with respect to  linear map $K:V\longrightarrow V$ consists of a pair $(\tilde{\rho},\tilde{\mu})$, where $(V;\tilde{\rho},\tilde{\mu})$ is a representation of the pre-Lie algebra $(\g,\cdot)$ such that for all $x\in \g, u\in V$, the following equalities are satisfied:
\begin{eqnarray}
\label{extension-rep-1}K(\tilde{\rho}(x)u)&=&\tilde{\rho}(x)K(u)+\tilde{\rho}(D(x))u,\\
\label{extension-rep-2}K(\tilde{\mu}(x)u)&=&\tilde{\mu}(x)K(u)+\tilde{\mu}(D(x))u.
\end{eqnarray}
\end{defi}
We denote a representation by $(V,K,\tilde{\rho},\tilde{\mu})$. Let $\tilde{L},\tilde{R}:\g\longrightarrow \gl(\g)$ be linear maps, where $\tilde{L}_xy=x\cdot y, \tilde{R}_xy=y\cdot x$. Then $(\g,D,\tilde{L},\tilde{R})$ is called a regular representation.

We define a bilinear operation $\cdot_\ltimes:\otimes^2(\g\oplus V)\lon(\g\oplus V)$ by
\begin{equation*}
(x+u)\cdot_\ltimes (y+v):=x\cdot y+\tilde{\rho}(x)(v)+\tilde{\mu}(y)(u),\quad \forall x,y \in \g,  u,v\in V.
\end{equation*}
and a linear map $D+K:\g\oplus V\longrightarrow \g\oplus V$ by
\begin{equation*}
(D+K)(x+u):=D(x)+K(u),\quad \forall x\in \g,  u\in V.
\end{equation*}

\begin{pro}\label{semi-direct}
 With the above notation, $(\g\oplus V,\cdot_\ltimes,D+K)$ is a regular pre-LieDer pair, which is denoted by $(\g\ltimes_{(\tilde{\rho},\tilde{\mu})}V,D+K)$ and called  the {\bf semi-direct product} of the regular pre-LieDer pair $(\g,D)$ and the representation $(V,K,\tilde{\rho},\tilde{\mu})$.
\end{pro}
\begin{proof}
Since $(V;\tilde{\rho},\tilde{\mu})$ is a representation of the pre-Lie algebra $(\g,\cdot)$, it is obviously that $(\g\oplus V,\cdot_\ltimes)$ is a pre-Lie algebra. For all $x,y\in\g, u,v\in V$, by \eqref{extension-rep-1} and \eqref{extension-rep-2}, we have
\begin{eqnarray*}
(D+K)((x+u)\cdot_\ltimes(y+v))&=&D(x\cdot y)+K(\tilde{\rho}(x)v)+K(\tilde{\mu}(y)u)\\
&=&D(x)\cdot y+x\cdot D(y)+\tilde{\rho}(x)K(v)+\tilde{\rho}(D(x))v+\tilde{\mu}(y)K(u)+\tilde{\mu}(D(y))u\\
&=&(x+u)\cdot_\ltimes(D+K)(y+v)+(D+K)(x+u)\cdot_\ltimes (y+v),
\end{eqnarray*}
which implies that $D+K$ is a derivation of the pre-Lie algebra $(\g\oplus V,\cdot_\ltimes)$. Thus, $(\g\oplus V,\cdot_\ltimes,D+K)$ is a regular pre-LieDer pair.
\end{proof}
Let $(V,K,\tilde{\rho},\tilde{\mu})$ be a representation of $(\g,D)$. Define the set of $0$-cochains $C^0(\g,D;V,K,\tilde{\rho},\tilde{\mu})$ to be $0$ and define the set of $1$-cochains $C^1(\g,D;V,K,\tilde{\rho},\tilde{\mu})$ to be $\Hom(\g,V)$. For $n\geq 2$, define the set of $n$-cochains $C^n(\g,D;V,K,\tilde{\rho},\tilde{\mu})$ by
\begin{equation*}
C^n(\g,D;V,K,\tilde{\rho},\tilde{\mu}):=\Hom(\wedge^{n-1}\g\otimes \g,V)\oplus \Hom(\wedge^{n-2}\g\otimes \g,V).
\end{equation*}
For all $f\in \Hom(\wedge^{n-1}\g\otimes \g,V)$ and $\theta\in \Hom(\wedge^{n-2}\g\otimes \g,V)$, define the coboundary operator $\huaD_{(\tilde{\rho},\tilde{\mu})}:C^n(\g,D;V,K,\tilde{\rho},\tilde{\mu})\longrightarrow C^{n+1}(\g,D;V,K,\tilde{\rho},\tilde{\mu})$ by
\begin{equation*}
\huaD_{(\tilde{\rho},\tilde{\mu})}(f,\theta)=(\dM_{(\tilde{\rho},\tilde{\mu})}f,\dM_{(\tilde{\rho},\tilde{\mu})}\theta+\Omega_{(\tilde{\rho},\tilde{\mu})}f),
\end{equation*}
where $\dM_{(\tilde{\rho},\tilde{\mu})}$ is the coboundary operator of the pre-Lie algebra $(\g,\cdot)$ with coefficients in the representation $(V,\tilde{\rho},\tilde{\mu})$ and $\Omega_{(\tilde{\rho},\tilde{\mu})}:\Hom(\wedge^{n-1}\g\otimes \g,V)\longrightarrow \Hom(\wedge^{n-1}\g\otimes \g,V)$ is defined by
\begin{equation*}
\Omega(f)(x_1,\dots,x_n)=(-1)^{n-2}(\Sigma_{i=1}^nf(x_1,\dots,x_{i-1},D(x_i),x_{i+1},\dots,x_n)-K(f(x_1,\dots,x_n))).
\end{equation*}
\begin{thm}
With the above notation, $(\oplus_{n=0}^{+\infty}C^n(\g,D;V,K,\tilde{\rho},\tilde{\mu}),\huaD_{(\tilde{\rho},\tilde{\mu})})$ is a cochain complex, i.e. $\huaD_{(\tilde{\rho},\tilde{\mu})}\circ \huaD_{(\tilde{\rho},\tilde{\mu})}=0$.
\end{thm}
\begin{proof}
By Proposition \ref{semi-direct}, we obtain that $(\g\ltimes_{(\rho,\mu)}V,D+K)$ is a regular pre-LieDer pair. By Theorem \ref{cochain complex}, $(\oplus_{n=0}^{+\infty}C^n(\g\oplus V,D+K),\bar{\huaD})$ is a cochain complex. It is straightforward to deduce that $(\oplus_{n=0}^{+\infty}C^n(\g,D;V,K,\tilde{\rho},\tilde{\mu}),\huaD_{(\tilde{\rho},\tilde{\mu})})$ is a subcomplex of $(\oplus_{n=0}^{+\infty}C^n(\g\oplus V,D+K),\bar{\huaD})$. Thus, we obtain that $\huaD_{(\tilde{\rho},\tilde{\mu})}\circ \huaD_{(\tilde{\rho},\tilde{\mu})}=0$.
\end{proof}
\begin{defi}
The cohomology of the cochain complex $(\oplus_{n=1}^{+\infty}C^n(\g,D;V,K,\tilde{\rho},\tilde{\mu}),\huaD_{(\tilde{\rho},\tilde{\mu})})$ is called the cohomology of the regular pre-LieDer pair $(\g,D)$ with coefficients in the representation $(V,K,\tilde{\rho},\tilde{\mu})$. The corresponding cohomology group is denoted by $H^n(\g,D;V,K,\tilde{\rho},\tilde{\mu})$.
\end{defi}

\begin{defi}
Let $(\g,\cdot,D)$ and $(V,\cdot_V,K)$ be two regular pre-LieDer pairs. An {\bf  extension} of $(\g,D)$ by $(V,K)$ is a short exact sequence of pre-LieDer pair morphisms:
$$\xymatrix{
  0 \ar[r] &V \ar[d]_{K}\ar[r]^{\iota}& \hat{\g}\ar[d]_{\hat{D}}\ar[r]^{p}&\g\ar[d]_{D}\ar[r]&0\\
     0\ar[r] &V \ar[r]^{ \iota} &\hat{\g}\ar[r]^{p} &\g\ar[r]&0,              }$$
where $(\hat{\g},\cdot_{\hat{\g}},\hat{D})$ is a regular pre-LieDer pair.

It is called an {\bf abelian extension} if  $(V,\cdot_V)$ is an abelian pre-Lie algebra, i.e.  for all $u,v\in V, u\cdot_V v=0$.
\end{defi}
\begin{defi}
A {\bf section} of an extension $(\hat{\g},\hat{D})$ of a regular pre-LieDer pair $(\g,D)$ by $(V,K)$ is a linear map $s:\g\longrightarrow \hat{\g}$ such that $p\circ s=\Id_{\g}$.
\end{defi}
Let $(\hat{\g},\hat{D})$ be an abelian extension of a regular pre-LieDer pair $(\g,D)$ by $(V,K)$ and $s:\g\longrightarrow \hat{\g}$ a section. For all $x,y\in \g$, define linear maps $\theta:\g\otimes \g\longrightarrow V$ and $\xi:\g\longrightarrow V$ respectively by
\begin{eqnarray*}
\theta(x,y)&=&s(x)\cdot_{\hat{\g}}s(y)-s(x\cdot y),\\
\xi(x)&=&\hat{D}(s(x))-s(D(x)).
\end{eqnarray*}
And for all $x,y\in \g, u\in V$, define $\tilde{\rho},\tilde{\mu}:\g\longrightarrow\gl(V)$ respectively by
\begin{eqnarray*}
\tilde{\rho}(x)(u)&=&s(x)\cdot_{\hat{\g}} u,\\
\tilde{\mu}(x)(u)&=&u\cdot_{\hat{\g}} s(x).
\end{eqnarray*}
Obviously, $\hat{\g}$ is isomorphic to $\g\oplus V$ as vector spaces. Transfer the regular pre-LieDer pair structure on $\hat{\g}$ to that on $\g\oplus V$, we obtain a regular pre-LieDer pair $(\g\oplus V,\diamond,\phi)$, where $\diamond$ and $\phi$ are given by
\begin{eqnarray*}
(x+u)\diamond (y+v)&=&x\cdot y+\theta(x,y)+\tilde{\rho}(x)(v)+\tilde{\mu}(y)(u),\quad \forall ~x,y\in \g, u,v\in V,\\
\phi(x+u)&=&D(x)+\xi(x)+K(u), \quad \forall ~x\in \g, u\in V.
\end{eqnarray*}

\begin{thm}\label{thm:representation}
With the above notation, $(V,K,\tilde{\rho},\tilde{\mu})$ is a representation of the regular pre-LieDer pair $(\g,D)$. Moreover, this representation is independent of the choice of sections.
\end{thm}
\begin{proof}
For all $x,y\in \g$, $u\in V$, by the definition of a pre-Lie algebra, we have
\begin{eqnarray*} 0&=&(x\diamond y)\diamond u-x\diamond (y\diamond u)-(y\diamond x)\diamond u+y\diamond (x\diamond u)\\
&=&(x\cdot y+\theta(x,y))\diamond u-x\diamond\tilde{\rho}(y)(u)-(y\cdot x+\theta(y,x))\diamond u+y\diamond\tilde{\rho}(x)(u)\\
&=&\tilde{\rho}(x\cdot y)u-\tilde{\rho}(x)\tilde{\rho}(y)u-\tilde{\rho}(y\cdot x)u+\tilde{\rho}(y)\tilde{\rho}(x)u,
\end{eqnarray*}
which implies that
\begin{equation}\label{representation-1}
\tilde{\rho}([x,y]_C)=\tilde{\rho}(x)\circ\tilde{\rho}(y)-\tilde{\rho}(y)\circ\tilde{\rho}(x).
\end{equation}
Similarly, we have
\begin{equation}\label{representation-2}
\tilde{\mu}(y)\circ\tilde{\mu}(x)-\tilde{\mu}(x\cdot y)=\tilde{\mu}(y)\circ\tilde{\rho}(x)-\tilde{\rho}(x)\circ\tilde{\mu}(y).
\end{equation}

For all $x \in \g$, $u\in V$, we have
\begin{eqnarray*}\nonumber 0&=&\phi(x\diamond u)-\phi(x)\diamond u-x\diamond\phi(u)\\
&=&\phi(\tilde{\rho}(x)u)-(D(x)+\xi(x))\diamond u-x\diamond K(u)\\
&=&K(\tilde{\rho}(x)u)-\tilde{\rho}(D(x))u-\tilde{\rho}(x)K(u),
\end{eqnarray*}
which implies that
\begin{equation}\label{representation-3}
K(\tilde{\rho}(x)u)=\tilde{\rho}(x)K(u)+\tilde{\rho}(D(x))u.
\end{equation}
Similarly, we have
\begin{equation}\label{representation-4}
K(\tilde{\mu}(x)u)=\tilde{\mu}(x)K(u)+\tilde{\mu}(D(x))u.
\end{equation}
Thus, By \eqref{representation-1}, \eqref{representation-2}, \eqref{representation-3} and \eqref{representation-4}, we obtain that $(V,K,\tilde{\rho},\tilde{\mu})$ is a representation.

Let $s'$ be another section and $(V,K,\tilde{\rho}',\tilde{\mu}')$ the corresponding representation of the regular pre-LieDer pair $(\g,D)$. Since $s(x)-s'(x)\in V$, then we have
\begin{equation*}
\tilde{\rho}(x)u-\tilde{\rho}'(x)u=(s(x)-s'(x))\cdot_{\hat{\g}}u=0,
\end{equation*}
which implies that $\tilde{\rho}=\tilde{\rho}'$. Similarly, we have $\tilde{\mu}=\tilde{\mu}'$.
Thus, this representation is independent of the choice of sections.
\end{proof}
\begin{thm}\label{thm:cocycle}
With the above notation, $(\theta,\xi)$ is a $2$-cocycle of the regular pre-LieDer pair $(\g,D)$ with coefficients in the representation $(V,K,\tilde{\rho},\tilde{\mu})$.
\end{thm}
\begin{proof}
For all $x,y,z\in \g$, by the definition of a pre-Lie algebra, we have
\begin{eqnarray*}0&=&(x\diamond y)\diamond z-x\diamond (y\diamond z)-(y\diamond x)\diamond z+y\diamond (x\diamond z)\\
&=&(x\cdot y+\theta(x,y)\big)\diamond z-x\diamond\big(y\cdot z +\theta(y,z)\big)-\big(y\cdot x+\theta(y,x)\big)\diamond z+y\diamond\big(x\cdot z +\theta(x,z))\\
&=&\theta(x\cdot y,z)+\tilde{\mu}(z)\theta(x,y)-\theta(x,y\cdot z)-\tilde{\rho}(x)\theta(y,z)\\
&&-\theta(y\cdot x,z)-\tilde{\mu}(z)\theta(y,x)+\theta(y,x\cdot z)+\tilde{\rho}(y)\theta(x,z)\\
&=&-\dM_{(\tilde{\rho},\tilde{\mu})}\theta(x,y)
\end{eqnarray*}
which implies that $\dM_{(\tilde{\rho},\tilde{\mu})}\theta=0.$
For all $x,y\in \g, u,v \in V$, we have
\begin{eqnarray*}0&=&\phi((x+u)\diamond (y+v))-\phi(x+u)\diamond (y+v)-(x+u)\diamond \phi(y+v)\\
&=&\phi(x\cdot y+\theta(x,y)+\tilde{\rho}(x)v+\tilde{\mu}(y)u)-(D(x)+\xi(x)+K(u))\diamond (y+v)-(x+u)\diamond (D(y)+\xi(y)+K(v))\\
&=&K(\theta(x,y))+\xi(x\cdot y)-\tilde{\mu}(y)\xi(x)-\theta(D(x),y)-\tilde{\rho}(x)\xi(y)-\theta(x,D(y))\\
&=&-(\dM_{(\tilde{\rho},\tilde{\mu})}\xi+\Omega_{(\tilde{\rho},\tilde{\mu})}\theta)(x,y),
\end{eqnarray*}
which implies that $\dM_{(\tilde{\rho},\tilde{\mu})}\xi+\Omega_{(\tilde{\rho},\tilde{\mu})}\theta=0.$

Thus, we obtain that $\huaD_{(\tilde{\rho},\tilde{\mu})}(\theta,\xi)=0$, which implies that $(\theta,\xi)$ is a $2$-cocycle of the the regular pre-LieDer pair $(\g,D)$ with coefficients in the representation $(V,K,\tilde{\rho},\tilde{\mu})$. The proof is finished.
\end{proof}
\begin{defi}\label{defi:isomorphic}
Let $(\hat{\g}_1,\cdot_{\hat{\g}_1},D_{\hat{\g}_1})$ and $(\hat{\g}_2,\cdot_{\hat{\g}_2},D_{\hat{\g}_2})$ be two abelian extensions of a regular pre-LieDer pair $(\g,D)$ by  $(V,K)$. They are said to be {\bf isomorphic} if there exists a regular pre-LieDer pair isomorphism $\zeta:(\hat{\g}_1,\cdot_{\hat{\g}_1},D_{\hat{\g}_1})\longrightarrow (\hat{\g}_2,\cdot_{\hat{\g}_2},D_{\hat{\g}_2})$, such that the following diagram is commutative:
$$\xymatrix{
  0 \ar[r] &V\ar @{=}[d]\ar[r]^{\iota_1}& \hat{\g}_1\ar[d]_{\zeta}\ar[r]^{p_1}&\g\ar @{=}[d]\ar[r]&0\\
     0\ar[r] &V\ar[r]^{ \iota_2} &\hat{\g}_2\ar[r]^{p_2} &\g\ar[r]&0.              }$$
\end{defi}

\begin{lem}
Let $(\hat{\g}_1,\cdot_{\hat{\g}_1},D_{\hat{\g}_1})$ and $(\hat{\g}_2,\cdot_{\hat{\g}_2},D_{\hat{\g}_2})$ be two isomorphic abelian extensions of a regular pre-LieDer pair $(\g,D)$ by $(V,K)$. Then they are give rise to the same representation of $(\g,D)$.
\end{lem}

\begin{proof}
Let $s_1:{\g}_1\longrightarrow \hat{\g}_1$ and $s_2:{\g}_2\longrightarrow \hat{\g}_2$ be two sections of $(\hat{\g}_1,\cdot_{\hat{\g}_1},D_{\hat{\g}_1})$ and $(\hat{\g}_2,\cdot_{\hat{\g}_2},D_{\hat{\g}_2})$ respectively. By Theorem \ref{thm:representation}, we obtain that $(V,K,\tilde{\rho}_1,\tilde{\mu}_1)$ and $(V,K,\tilde{\rho}_2,\tilde{\mu}_2)$ are their representations respectively. Define $s'_1:{\g}_1\longrightarrow \hat{\g}_1$ by $s'_1=\zeta^{-1}\circ s_2$. Since $\zeta:(\hat{\g}_1,\cdot_{\hat{\g}_1},D_{\hat{\g}_1})\longrightarrow (\hat{\g}_2,\cdot_{\hat{\g}_2},D_{\hat{\g}_2})$ is a pre-LieDer pair isomorphism satisfying the commutative diagram in Definition \ref{defi:isomorphic}, by $p_2\circ \zeta=p_1$, we have
\begin{equation*}
p_1\circ s'_1=p_2\circ \zeta \circ \zeta^{-1}\circ s_2=\Id_{\g}.
\end{equation*}
Thus, we obtain that $s'_1$ is a section of $(\hat{\g}_1,\cdot_{\hat{\g}_1},D_{\hat{\g}_1})$. For all $x\in \g, u\in V$, we have
\begin{equation*}
\tilde{\rho}_1(x)(u)=s'_1(x) \cdot_{\hat{\g}_1} u=(\zeta^{-1}\circ s_2)(x)\cdot_{\hat{\g}_1} u=\zeta^{-1}(s_2(x) \cdot_{\hat{\g}_2} u)=\tilde{\rho}_2(x)(u),
\end{equation*}
which implies that $\tilde{\rho}_1=\tilde{\rho}_2$.
Similarly, we have $\tilde{\mu}_1=\tilde{\mu}_2$. This finishes the proof.
\end{proof}
So in the sequel, we fixed a representation $(V,K,\tilde{\rho},\tilde{\mu})$ of a regular pre-LieDer pair $(\g,D)$ and consider abelian extensions that induce the given representation.

\begin{thm}
Abelian extensions of a regular pre-LieDer pair $(\g,D)$ by $(V,K)$ are classified by the second cohomology group $H^2(\g,D;V,K,\tilde{\rho},\tilde{\mu})$.
\end{thm}
\begin{proof}
Let $(\hat{\g},\hat{D})$ be an abelian extension of a regular pre-LieDer pair $(\g,D)$ by $(V,K)$. Choosing a section $s:\g\longrightarrow \hat{\g}$, by Theorem \ref{thm:cocycle}, we obtain that $(\theta,\xi)\in Z^2(\g,D;V,K,\tilde{\rho},\tilde{\mu})$. Now we show that the cohomological class of $(\theta,\xi)$ does not depend on the choice of sections. In fact, let $s$ and $s'$ be two different sections. Define $\varphi:\g\longrightarrow V$ by $\varphi(x)=s(x)-s'(x)$. Then for all $x,y\in \g$, we have
\begin{eqnarray*}
 \theta(x,y)&=&s(x)\cdot_{\hat{\g}} s(y)-s(x\cdot y)\\
  &=&\big(s'(x)+\varphi(x)\big)\cdot_{\hat{\g}} \big(s'(y)+\varphi(y)\big)-s'(x\cdot y)-\varphi(x\cdot y)\\
  &=&s'(x)\cdot_{\hat{\g}} s'(y)+\tilde{\rho}(x)\varphi(y)+\tilde{\mu}(y)\varphi(x)-s'(x\cdot y)-\varphi(x\cdot y)\\
  &=&\theta'(x,y)+\dM_{(\tilde{\rho},\tilde{\mu})}\varphi(x,y),
\end{eqnarray*}
which implies that $\theta-\theta'=\dM_{(\tilde{\rho},\tilde{\mu})}\varphi$.
Similarly, we have $\xi-\xi'=\Omega_{(\tilde{\rho},\tilde{\mu})}\varphi$.

Therefore, we obtain that $(\theta-\theta',\xi-\xi')=(\dM_{(\tilde{\rho},\tilde{\mu})}\varphi,\Omega_{(\tilde{\rho},\tilde{\mu})}\varphi)=\huaD_{(\tilde{\rho},\tilde{\mu})}\varphi$, $(\theta,\xi)$ and $(\theta',\xi')$ are in the same cohomological class.

Now we go on to prove that isomorphic abelian extensions give rise to the same element in  $H^2(\g,D;V,K,\tilde{\rho},\tilde{\mu})$. Assume that $(\hat{\g}_1,\cdot_{\hat{\g}_1},D_{\hat{\g}_1})$ and $(\hat{\g}_2,\cdot_{\hat{\g}_2},D_{\hat{\g}_2})$ are two isomorphic abelian extensions of a regular pre-LieDer pair $(\g,D)$ by $(V,K)$, and $\zeta:(\hat{\g}_1,\cdot_{\hat{\g}_1},D_{\hat{\g}_1})\longrightarrow (\hat{\g}_2,\cdot_{\hat{\g}_2},D_{\hat{\g}_2})$ is a pre-LieDer pair isomorphism satisfying the commutative diagram in Definition \ref{defi:isomorphic}. Assume that $s_1:\g\longrightarrow \hat{\g}_1$ is a section of $\hat{\g}_1$. By $p_2\circ \zeta=p_1$, we have
\begin{equation*}
p_2\circ (\zeta\circ s_1)=p_1\circ s_1=\Id_{\g}.
\end{equation*}
Thus, we obtain that $\zeta\circ s_1$ is a section of $\hat{\g}_2$. Define $s_2=\zeta\circ s_1$. Since $\zeta$ is an isomorphism of pre-LieDer pair and $\zeta\mid_V=\Id_V$, for all $x,y\in \g$, we have
\begin{eqnarray*}
 \theta_2(x,y)&=&s_2(x)\cdot_{\hat{\g}_2} s_2(y)-s_2(x\cdot y)\\
  &=&(\zeta\circ s_1)(x)\cdot_{\hat{\g}_2}(\zeta\circ s_1)(y)-(\zeta\circ s_1)(x\cdot y)\\
  &=&\zeta\big(s_1(x)\cdot_{\hat{\g}_1} s_1(y)-s_1(x\cdot y)\big)\\
  &=&\theta_1(x,y),
\end{eqnarray*}
Similarly, we have $\xi_1=\xi_2$.
Thus, isomorphic abelian extensions gives rise to the same element in $H^2(\g,D;V,K,\tilde{\rho},\tilde{\mu})$.

Conversely, given two 2-cocycles $(\theta_1,\xi_1)$ and $(\theta_2,\xi_2)$, we can construct two abelian extensions $(\g\oplus V,\diamond_1,\phi_1)$ and $(\g\oplus V,\diamond_2,\phi_2)$. If  $(\theta_1,\xi_1), (\theta_2,\xi_2)\in H^2(\g,D;V,K,\tilde{\rho},\tilde{\mu})$, then there exists $\varphi:\g\longrightarrow V$, such that $\theta_1=\theta_2+\dM_{(\tilde{\rho},\tilde{\mu})}\varphi$ and $\xi_1=\xi_2+\Omega_{(\tilde{\rho},\tilde{\mu})}\varphi$. We define $\zeta:\g\oplus V\longrightarrow \g\oplus V$ by
\begin{equation*}
\zeta(x+u)=x+u+\varphi(x),\quad \forall ~x\in \g, u\in V.
\end{equation*}
For all $x,y\in \g, u,v\in V$, by $\theta_1=\theta_2+\dM_{(\tilde{\rho},\tilde{\mu})}\varphi$, we have
\begin{eqnarray*}&&\zeta\big((x+u)\diamond_1(y+v)\big)- \zeta(x+u)\diamond_2\zeta(y+v)\\
&=&\zeta\big(x\cdot y+\theta_1(x,y)+\tilde{\rho}(x)(v)+\tilde{\mu}(y)(u)\big)-\big(x+u+\varphi(x)\big)\diamond_2 \big(y+v+\varphi(y)\big)\\
&=&\theta_1(x,y)+\varphi(x\cdot y)-\theta_2(x,y)-\tilde{\rho}(x)\varphi(y)-\tilde{\mu}(y)\varphi(x)\\
&=&\theta_1(x,y)-\theta_2(x,y)-\dM_{(\tilde{\rho},\tilde{\mu})}\varphi(x,y)\\
&=&0,
\end{eqnarray*}
and for all $x\in \g, u\in V$, by $\xi_1=\xi_2+\Omega_{(\tilde{\rho},\tilde{\mu})}\varphi$, we have
\begin{eqnarray*}&&\zeta \circ\phi_1(x+u)-\phi_2\circ \zeta(x+u)\\
&=&\zeta\big(D(x)+\xi_1(x)+K(u)\big)-\phi_2\big(x+u+\varphi(x)\big)\\
&=&\xi_1(x)+\varphi(D(x))-\xi_2(x)-K(\varphi(x))\\
&=&\xi_1(x)-\xi_2(x)-\Omega_{(\tilde{\rho},\tilde{\mu})}\varphi(x)\\
&=&0,
\end{eqnarray*}
Thus, $\zeta$ is a pre-LieDer pair isomorphism from $(\g\oplus V,\diamond_1,\phi_1)$ to $(\g\oplus V,\diamond_2,\phi_2)$. Moreover, it is obvious that the diagram in Definition \ref{defi:isomorphic} is commutative. This finishes the proof.
\end{proof}

\noindent{\bf Acknowledgement:}  This work is  supported by  NSF of Jilin Province (No. YDZJ202201ZYTS589), NNSF of China (Nos. 12271085, 12071405) and the Fundamental Research Funds for the Central Universities.

 \end{document}